\documentclass[12pt]{article}
\usepackage[margin=2cm]{geometry}
\usepackage{epsfig,color}
\usepackage{amsmath,amssymb,amsthm}
\usepackage{enumerate}
\newtheorem{theorem}{Theorem}[section]
\newtheorem{definition}{Definition}[section]

\newtheorem{corollary}{Corollary}[section]

\newcommand{\noi}{\noindent}

\newcommand{\mZ}{{\mathbb Z}}
\title{\textbf{Distance Antimagic Labeling of Zero-Divisor Graphs}}

\author{\textbf{V. Sivakumaran$^{1,} $\thanks{Corresponding author: anukumar675@gmail.com}, K. Sankar$^2$, S. Prabhu$^{1} $}\\
$^1$\small Department of Mathematics, Rajalakshmi Engineering College, Thandalam, Chennai 602105, India\\
$^2$\small Department of Mathematics, Anna University, Chennai 600025, India}
\date{}
\begin{document}
\maketitle
\begin{abstract}
In this paper, we prove that for all $m\geq 1$ and $n=1$, the graph $ m\Gamma(\mZ_9)+n\Gamma(\mZ_4)$, for all $n\geq 1$, and $m=1$, the graph $ m\overline{\Gamma(\mZ_6)}+n\Gamma(\mZ_9)$, for all $m\geq1$, $[m\Gamma(\mZ_9)+\Gamma(\mZ_4)]\times \Gamma(\mZ_9)$, for all prime $m\geq3$, $\Gamma(\mZ_6)\times\Gamma(\mZ_{2m})$ and $\Gamma(\mZ_6)\times\Gamma(\mZ_{m^2})$ are all admit distance antimagic labeling.

\noi\textbf{Keywords:} Distance antimagic labeling, Zero-divisor
graph, Join of graphs, Cartesian product of graphs.

\noi\textbf{Mathematics Subject Classification:} 05C78.
\end{abstract}
\section{Introduction}
In the year 1990, Beck \cite{be88} introduced Zero-divisor graph of a commutative ring. Anderson and Livingston studied and developed the relation between theoretical properties of the ring and graph theoretical properties of zero-divisor graph $\Gamma(R)$ in \cite{al99}. 

Consider a commutative ring $R$ with unity and assume set of all zero-divisor of $R$ be $\mZ(R)$ and $\mZ^*(R)=\mZ(R)\setminus \{0\}$ be the vertex set such that any two distinct vertices $u,v \in \mZ^*(R)$ are said to be adjacent if and only if $uv=0$. This condition provides zero-divisor graph and is denoted by $\Gamma(R)$. A graph $\Gamma(\mZ_{10})$ with $V[\Gamma(\mZ_{10})] = \{2, 4, 6, 8, 5\}$ and edge set
$E[\Gamma(\mZ_{10})] = \{(2, 5), (4, 5), (6, 5), (8, 5)\}$ is given in
Figure \ref{fig1}. A graph $\Gamma(\mZ_{15})$ with $V[\Gamma(\mZ_{15})] = \{3, 6, 9, 12, 5, 10\}$ and edge set
$E[\Gamma(\mZ_{15})] = \{(3, 5)$, $(6, 5)$, $(9, 5)$, $(12, 5)$, $(3,
10)$, $(6, 10)$, $(9, 10)$, $(12, 10)\}$ is given in Figure
\ref{fig2}.

\begin{figure}[h!]
  \centerline{\includegraphics[scale=1]{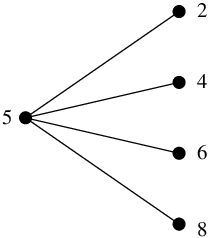}}
  \caption{$\Gamma(\mZ_{10})$}
  \label{fig1}
\end{figure}

\begin{figure}[h!]
  \centerline{\includegraphics{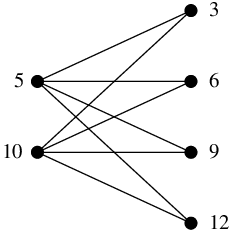}}
  \caption{$\Gamma(\mZ_{15})$}
  \label{fig2}
\end{figure}

\begin{definition}{\rm \cite{we01}}
Let $H_1, H_2, \dots, H_m$ be the given $m$ copies of disjoint graphs. Then a set of $k$ disjoint graphs.  Then $H_1+H_2+ \dots +H_m$ is obtained from $H_1, H_2, \dots, H_m$ by joining every vertex of $H_i$ with every vertex $H_j$, whenever $i\neq j$. 
{\it join graph} $G_1 + G_2 + \ldots + G_k$ is obtained from $G_1,
G_2, \dots, G_k$ by joining every vertex of $G_i$ with every vertex of
$G_j$, whenever $i \neq j$.
\end{definition}

\begin{definition}{\rm \cite{we01}}
The {\it Cartesian product} of the graphs $G$ and $H$, denoted $G
\times H$, is the graph with $V(G \times H) = \{(u, v) : u
\in V(G) \text{ and } v \in V(H)\}$ and $E(G \times H) =
\{<(u, v), (u', v')> : u = u' \text{ and } vv' \in E(H) \text{ or } v
= v' \text{ and } uu' \in E(G)\}$
\end{definition}
If a graph $G$ is said to follow a DAML if for any two distinct vertices $u_1, u_2 \in V(G)$ such that $N(u_1)\neq N(u_2)$. 
\section{Main Results}
In the main result, we show that some join of zero-divisor this paper, we prove that some join of zero-divisor graphs and some
Cartesian product of zero-divisor graphs are distance antimagic
graphs.

\begin{theorem}\label{thm2.1}
The join graph $\Gamma(\mZ_{2m}) + \Gamma(\mZ_4)$ does not have DAML
for all prime $m \geq 3$.
\end{theorem}

\begin{proof}
Consider the graphs $\Gamma(\mZ_{2m})$ and
$\Gamma(\mZ_4)$ for all prime numbers $m \geq 3$.\\
Let $G = \Gamma(\mZ_{2m}) + \Gamma(\mZ_4)$ be the join of graphs as
shown in Figure \ref{fig4}.\\

There exist  $u, v \in V(G)$ with $N(u) = N(v)$.  By the
necessary condition for distance antimagic graphs, the graph
$\Gamma(\mZ_{2m}) + \Gamma(\mZ_4)$ does not have DAML for all prime $m
\geq 3$.

\begin{figure}[h!]
  \centerline{\includegraphics{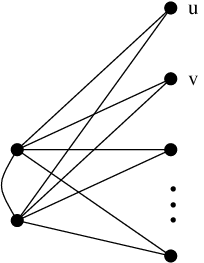}}
  \caption{The graph $\Gamma(\mZ_{2m}) + \Gamma(\mZ_4)$}
  \label{fig4}
\end{figure}

\end{proof}

We state the following few theorems without proof.  The proof of those
theorems can be done as the proof of Theorem \ref{thm2.1}.
\begin{theorem}
	For every prime $m\geq 3$, $\Gamma(\mZ_{2m})+G$, $G\in \{\Gamma(\mZ_6),\Gamma(\mZ_9), \overline{\Gamma(\mZ_6)},  \overline{\Gamma(\mZ_9)}\}$ does not admit DAML.
\end{theorem}
\begin{theorem}
	For every prime $m\geq 5$, $\Gamma(\mZ_{3m})+G$, $G\in \{\Gamma(\mZ_4), \Gamma(\mZ_6), \Gamma(\mZ_9),  \overline{\Gamma(\mZ_6)}, \overline{\Gamma(\mZ_9)}\}$ does not admit DAML.
\end{theorem}

\begin{theorem}
	For all $m\geq 1$ and $n=1$ if and only if the graph $m\Gamma(\mZ_9)+n\Gamma(\mZ_4)$ admits DAML.
\end{theorem}
\begin{proof}
Consider the zero-divisor graphs $\Gamma(\mZ_9)$ and $\Gamma(\mZ_4)$.\\
Let $G = m \Gamma(\mZ_9) + n \Gamma(\mZ_4)$ with $|V(G)| = 2m+n$ for
all $m \geq 1$, and $n \geq 1$.

\noi\textbf{Case (i):} For $n = 1$ and $m \geq 1$.

\noi Consider the graph $m \Gamma(\mZ_9) + \Gamma(\mZ_4)$ with $|V(m
\Gamma(\mZ_9) + \Gamma(\mZ_4))| = 2m+1$.\\
We label the vertices of $m \Gamma(\mZ_9) + \Gamma(\mZ_4)$ as shown in
Figure \ref{fig5}. 

\begin{figure}[h!]
  \centerline{\includegraphics{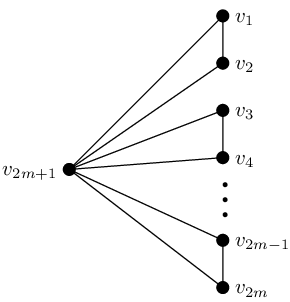}}
  \caption{The graph $m \Gamma(\mZ_9) + \Gamma(\mZ_4)$}
  \label{fig5}
\end{figure}

Define $f : V(m \Gamma(\mZ_9) + \Gamma(\mZ_4)) \to \{1, 2, \dots,
2m+1\}$ by\\ $f(v_i) = i$, whenever $i \in \{1,2,\ldots, 2m+1\}$ .

The weight of each vertex is determined by:\\
$w(v_i) = \begin{cases} 2m+2+i, & \text{if } i \text{ is odd} \\ 2m+i, 
  & \text{if } i \text{ is even} \end{cases}$\\  
and $w(v_{2m+1}) = m(2m+1)$\\
It follows that $w(v_i) \neq w(v_j)$ for all $i \neq j$.

Hence, the graph $m \Gamma(\mZ_9) + \Gamma(\mZ_4)$ admits DAML for all $m \geq 1$.

\noi\textbf{Case (ii):} For $n > 1$ and $m \geq 1$.

\noi Consider the graph $m \Gamma(\mZ_9) + n \Gamma(\mZ_4)$ with $|V(m
\Gamma(\mZ_9) + n \Gamma(\mZ_4))| = 2m+n$.  We label the vertices of
$m \Gamma(\mZ_9) + n \Gamma(\mZ_4)$ as shown in Figure \ref{fig6}.

\begin{figure}[h!]
  \centerline{\includegraphics{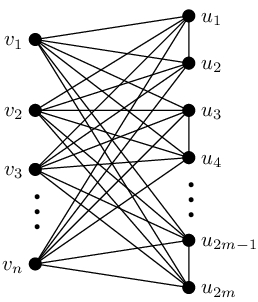}}
  \caption{$m \Gamma(\mZ_9) + n \Gamma(\mZ_4)$}
  \label{fig6}
\end{figure}

Clearly, we have $N(v_1) = N(v_2) = \dots = N(v_m)$.\\
Hence, the graph $m \Gamma(\mZ_9) + n \Gamma(\mZ_4)$ does not have
DAML for all $m \geq 1$ and $n > 1$.

Therefore, the graph $m \Gamma(\mZ_9) + n \Gamma(\mZ_4)$ admits DAML if
and only if for $m \geq 1$ and $n = 1$.
\end{proof}

\begin{theorem}
For all $m\geq1$, and $n\geq 1$, the graph  $m \Gamma(\mZ_6) + n \Gamma(\mZ_4)$ does not admit DAML. 
\end{theorem}

\begin{proof}
Consider the zero-divisor graphs $\Gamma(\mZ_6)$ and
$\Gamma(\mZ_4)$.\\ 
Let $G = m \Gamma(\mZ_6) + n \Gamma(\mZ_4)$ with $|V(G)| = 3m+n$, for $m \geq 1$ and $n \geq 1$.\\
We label the vertices of $G$ as in Figure \ref{fig7}.

\begin{figure}[h!]
  \centerline{\includegraphics{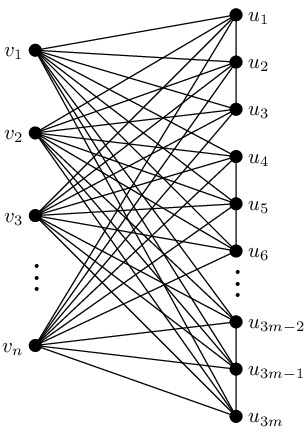}}
  \caption{$m \Gamma(\mZ_6) + n \Gamma(\mZ_4)$}
  \label{fig7}
\end{figure}

\noi We observe that $N(u_1) = N(u_3)$, $N(u_4) = N(u_6),\ldots, N(u_{3m-2}) = N(u_{3m})$.\\
Hence, the graph $m \Gamma(\mZ_6) + n \Gamma(\mZ_4)$ does not
admit DAML for all $m \geq 1$ and $n \geq 1$.
\end{proof}

\begin{theorem}
For all $m \geq 1$ and $n \geq 1$,  $m \Gamma(\mZ_6) + n \Gamma(\mZ_9)$ does not admit DAML.
\end{theorem}
\begin{figure}[h!]
	\centerline{\includegraphics{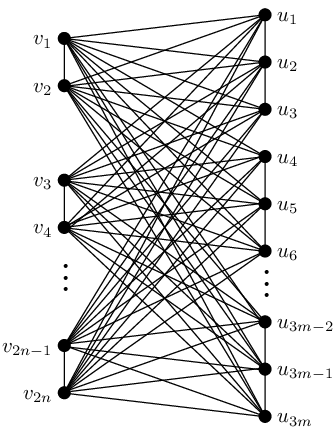}}
	\caption{ $m \Gamma(\mZ_6) + n \Gamma(\mZ_9)$}
	\label{fig8}
\end{figure}

\begin{proof}
Consider the zero-divisor graphs $\Gamma(\mZ_6)$ and
$\Gamma(\mZ_9)$. Let $G = m \Gamma(\mZ_6) + n \Gamma(\mZ_9)$ with $|V(G)| = 3m+2n$, for
 $m,n \geq 1$. The vertices of $G$  labeled as in Figure \ref{fig8}. We have  $N(u_1) = N(u_3)$, $\dots$, $N(u_{3m-2}) =
N(u_{3m})$. Hence, the graph $m \Gamma(\mZ_6) + n \Gamma(\mZ_9)$ does not admit DAML for all $m \geq 1$ and $n \geq 1$.
\end{proof}

\begin{corollary}
For all  $m \geq 1$, $m \Gamma(\mZ_6) + \overline{\Gamma(\mZ_9)}$ does not admit  DAML.
\end{corollary}

\begin{proof}
Let $G = m \Gamma(\mZ_6) + \overline{\Gamma(\mZ_9)}$ be a join graph and $|V(G)| = 3m+2$. There exist $u,v \in \overline{\Gamma(\mZ_9)}(G)$, we have $N(u) = N(v) = 3m$.
Hence, the graph $m \Gamma(\mZ_6) + \overline{\Gamma(\mZ_9)}$ does not admit distance antimagic labeling.
\end{proof}

\begin{theorem}
The join graph $m \overline{\Gamma(\mZ_6)} + n \Gamma(\mZ_9)$ admits DAML
if and only if $m = 1$ and $n \geq 1$.
\end{theorem}

\begin{proof}
Consider the zero-divisor graphs $\Gamma(\mZ_6)$ and
$\Gamma(\mZ_9)$.\\ 
Let $G = m \overline{\Gamma(\mZ_6)} + n \Gamma(\mZ_9)$ with $|V(G)| = 3m+2n$ for all $m \geq 1$ and $n \geq 1$.

For $m > 1$ and $n \geq 1$, Consider the graph $m \overline{\Gamma(\mZ_6)} + n \Gamma(\mZ_9)$ with
$|V(m \overline{\Gamma(\mZ_6)} + n \Gamma(\mZ_9))| = 3m+2n$.
We label the vertices of the graph $m \overline{\Gamma(\mZ_6)} + n\Gamma(\mZ_9)$ as in Figure \ref{fig9}.

\begin{figure}[h!]
  \centerline{\includegraphics{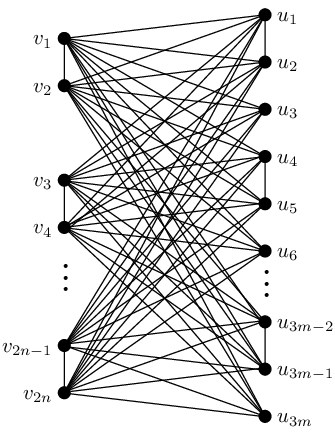}}
  \caption{The graph $m \overline{\Gamma(\mZ_6)} + n \Gamma(\mZ_9)$}
  \label{fig9}
\end{figure}

It is observed that $N(u_3) = N(u_6) = \dots = N(u_{3m})$.
Hence, the graph $m \overline{\Gamma(\mZ_6)} + n \Gamma(\mZ_9)$
does not admit DAML, for all $m > 1$ \& $n \geq 1$.

For $m = 1$ and $n \geq 1$, consider the graph $\overline{\Gamma(\mZ_6)} + n \Gamma(\mZ_9)$ with $|V(\overline{\Gamma(\mZ_6)} + n \Gamma(\mZ_9))| = 2n+3$.  
Let $N = 2n$ and the vertices of $G$ is labeled as shown in Figure
\ref{fig10}.

\begin{figure}[h!]
  \centerline{\includegraphics{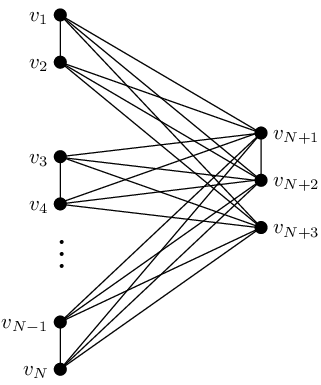}}
  \caption{The graph $\overline{\Gamma(\mZ_6)} + n \Gamma(\mZ_9)$}
  \label{fig10}
\end{figure}
Let $f : V(G) \to \{1, 2, \dots, N+3\}$ be a labeling function and is defined by $f(v_i) = i$, if $1 \leq i \leq N+3$. The weight of each vertex is determined by \\
$w(v_i) = \begin{cases} 3N+i+7, & \text{if } i \text{  odd and } 1
  \leq i \leq N-1 \\ 3N+i+5, & \text{if } i \text{  even and } 2
  \leq i \leq N \end{cases}$\\
$w(v_{N+1}) = (N^2+3N+4)/2$,\\
$w(v_{N+2}) = (N^2+3N+2)/2$ and\\
$w(v_{N+3}) = (N^2+N)/2$.\\
It follows that $w(v_i) \neq w(v_j)$ for all $i \neq j$.\\
Then the graph $\overline{\Gamma(\mZ_6)} + n \Gamma(\mZ_9)$ admits DAML
for all $n \geq 1$.\\
Hence, the graph $m \overline{\Gamma(\mZ_6)} + n \Gamma(\mZ_9)$
admits DAML if and only if $m = 1$ and $n \geq 1$.
\end{proof}

\begin{theorem}
The graph $[m \Gamma(\mZ_9) + \Gamma(\mZ_4)] \times \Gamma(\mZ_9)$ has
distance antimagic labeling for all $m \geq 1$.
\end{theorem}

\begin{proof}
Let $G = [m \Gamma(\mZ_9) + \Gamma(\mZ_4)] \times \Gamma(\mZ_9)$ be a
cartesian product graph with $|V(G)| = 4m+2$ for $m \geq 1$.
We prove this theorem in two cases:

\noi{\bf Case 1.} Let $m = 1$, then the graph $[\Gamma(\mZ_9)
+ \Gamma(\mZ_4)] \times \Gamma(\mZ_9)$ be a Cartesian product graph of
order 6. We label the vertices of $[\Gamma(\mZ_9)
+ \Gamma(\mZ_4)] \times \Gamma(\mZ_9)$ as given in
Figure \ref{fig6.10}.

\begin{figure}[h!]
  \centerline{\includegraphics{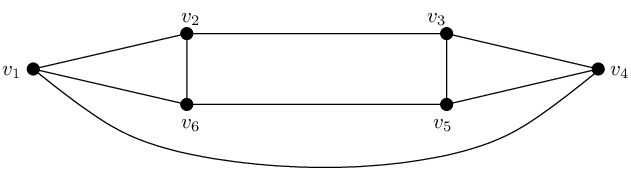}}
  \caption{The Cartesian product graph $[\Gamma(\mZ_9)
  + \Gamma(\mZ_4)] \times \Gamma(\mZ_9)$}
  \label{fig6.10}
\end{figure}
Let $f : V([\Gamma(\mZ_9) + \Gamma(\mZ_4)] \times \Gamma(\mZ_9)) \to \{1,
2, 3, 4, 5, 6\}$ be a labeling function and defined by\\
$f(v_1) = 1$, $f(v_2) = 2$, $f(v_3) = 3$, $f(v_4) = 4$, $f(v_5) = 5$,
$f(v_6) = 6$\\
The weight of each vertex is determined by\\
$w(v_1) = 12$, $w(v_2) = 10$, $w(v_3) = 11$, $w(v_4) = 9$, $w(v_5) =
13$, $w(v_6) = 8$

It follows that the graph $[\Gamma(\mZ_9)
+ \Gamma(\mZ_4)] \times \Gamma(\mZ_9)$ has distance antimagic
labeling.  Also we observe that the graph $[\Gamma(\mZ_9)
+ \Gamma(\mZ_4)] \times \Gamma(\mZ_9)$ is a (8,1)-distance antimagic
graph.

\noi{\bf Case 2.} Let $m \geq 2$, then the graph $[m\Gamma(\mZ_9)
+ \Gamma(\mZ_4)] \times \Gamma(\mZ_9)$ be a Cartesian product graph of
order $4m+2$.\\
We label the vertices of $[m\Gamma(\mZ_9)
+ \Gamma(\mZ_4)] \times \Gamma(\mZ_9)$ as depicted in
Figure \ref{fig6.11}.

\begin{figure}[h!]
  \centerline{\includegraphics[scale=0.9]{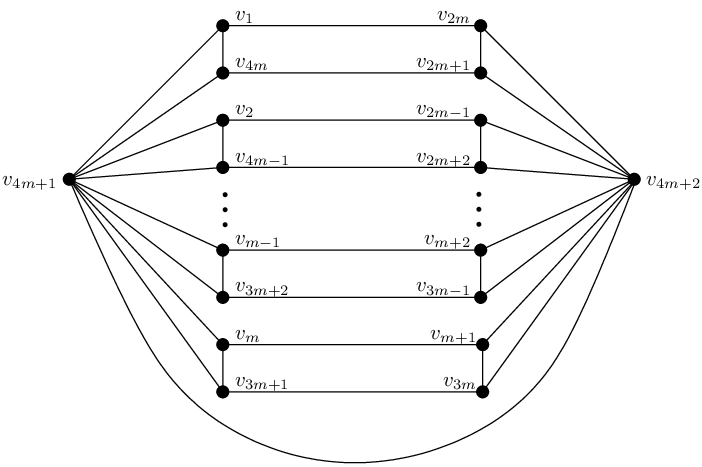}}
  \caption{The Cartesian product graph $[m \Gamma(\mZ_9)
  + \Gamma(\mZ_4)] \times \Gamma(\mZ_9)$}
  \label{fig6.11}
\end{figure}
Let
$f : V([m\Gamma(\mZ_9) + \Gamma(\mZ_4)] \times \Gamma(\mZ_9)) \to \{1,
2, \dots, 4m+2\}$ be a labeling function and defined by\\
$f(v_i) = i$, if $1 \leq i \leq 4m+2$.\\
The weight of each vertex is determined by\\
$w(v_i) = 10m+3-2i$, if $1 \leq i \leq m$\\
$w(v_i) = 10m+4-2i$, if $m+1 \leq i \leq 2m$\\
$w(v_i) = 14m+4-2i$, if $2m+1 \leq i \leq 3m$\\
$w(v_i) = 14m+3-2i$, if $3m+1 \leq i \leq 4m$\\
$w(v_{4m+1}) = 4m^2+5m+2$\\
$w(v_{4m+2}) = 4m^2+5m+1$, for $m \geq 2$\\
Observe that $w(v_i) \neq w(v_j)$, for all $i \neq j$.\\
Therefore, the graph $[m \Gamma(\mZ_9) + \Gamma(\mZ_4)] \times
\Gamma(\mZ_9)$ has distance antimagic labeling for all $m \geq 2$.
Hence, from Case 1 and 2, the graph $[m \Gamma(\mZ_9)
+ \Gamma(\mZ_4)] \times \Gamma(\mZ_9)$ has distance antimagic labeling
for all $m \geq 1$.
\end{proof}

\begin{theorem}
The graph $\Gamma(\mZ_6) \times \Gamma(\mZ_{2m})$ admits DAML for all
prime $m \neq 2$.
\end{theorem}

\begin{proof}
Consider the zero-divisor graphs $\Gamma(\mZ_6)$ and
$\Gamma(\mZ_{2m})$.\\
Let $G = \Gamma(\mZ_6) \times \Gamma(\mZ_{2m})$ with $|V(G)| = 3m$,
for all prime $m \neq 2$.

\noi{\bf Case (i):} For $m = 3$, $|V(\Gamma(\mZ_6) \times \Gamma(\mZ_6))| =9$.

We label the vertices of $\Gamma(\mZ_6) \times \Gamma(\mZ_6)$ as shown
in Figure \ref{fig12}.

\begin{figure}[h!]
  \centerline{\includegraphics{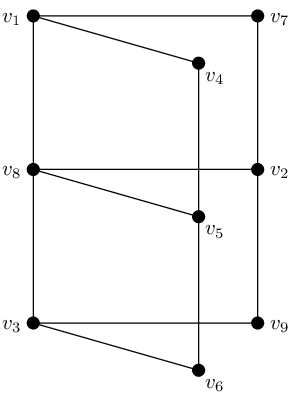}}
  \caption{The graph $\Gamma(\mZ_6) \times \Gamma(\mZ_6)$}
  \label{fig12}
\end{figure}

Let
$f : V(G) \to \{1, 2, \dots, 9\}$ be a labeling function and is defined by \\
$f(v_1) = 1$, $f(v_2) = 2$, $f(v_3) = 3$, $f(v_4) = 4$, $f(v_5) =
5$,\\
$f(v_6) = 6$, $f(v_7) = 7$, $f(v_8) = 8$, $f(v_9) = 9$.\\
The weight of each vertex is determined by:\\
$w(v_1) = 19$, $w(v_2) = 24$, $w(v_3) = 23$,\\
$w(v_4) = 6$, $w(v_5) = 18$, $w(v_6) = 8$,\\
$w(v_7) = 3$, $w(v_8) = 11$, $w(v_9) = 5$.\\
It follows that the graph $\Gamma(\mZ_6) \times \Gamma(\mZ_6)$ admits DAML.

\noi{\bf Case (ii):} Let $m > 3$ and prime.  

\noi We assume the graph $\Gamma(\mZ_6) \times \Gamma(\mZ_{2m})$ with
$|V(\Gamma(\mZ_6) \times \Gamma(\mZ_{2m}))| = 3m$.\\
We label the vertices of $\Gamma(\mZ_6) \times \Gamma(\mZ_{2m})$ as
shown in Figure \ref{fig13}.

\begin{figure}[h!]
  \centerline{\includegraphics{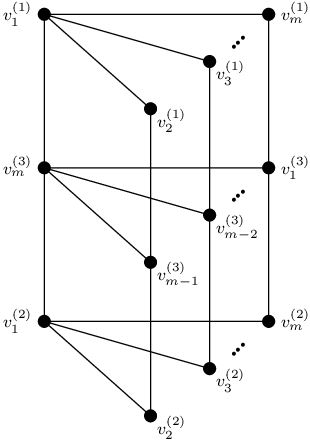}}
  \caption{The graph $\Gamma(\mZ_6) \times \Gamma(\mZ_{2m})$}
  \label{fig13}
\end{figure}

$f : V(G) \to \{1, 2, \dots, 3m\}$ be a labeling function defined by\\
$f(v_i^{(j)}) = \begin{cases} i, & \text{where }  \ j = 1  \ \text{and}\ 1 \leq i \leq m\\ 
                            m+i, & \text{where}  \  j = 2 \ \text{and} \ 1 \leq i \leq m\\    
                            2m+i, & \text{where} \ j = 3 \ \text{and}\ 1 \leq i \leq m \end{cases}$\\
The weight of each vertex is given by:\\
$w(v_1^{(1)}) = (m^2+7m-2)/2$,\\
$w(v_i^{(1)}) = 3m-i+2$, if $2 \leq i \leq m$,\\
$w(v_1^{(2)}) = (3m^2+5m-2)/2$,\\
$w(v_i^{(2)}) = 4m-i+2$, if $2 \leq i \leq m$,\\
$w(v_i^{(3)}) = 6m-2i+2$, if $1 \leq i \leq m-1$ and\\
$w(v_m^{(3)}) = (5m^2-3m+4)/2$.\\
Therefore, the graph $\Gamma(\mZ_6) \times \Gamma(\mZ_{2m})$ admits DAML for
all prime $m > 3$.\\
Hence, the graph $\Gamma(\mZ_6) \times \Gamma(\mZ_{2m})$ admits DAML for
all prime $m \neq 2$.
\end{proof}

\begin{theorem}
The graph $\Gamma(\mZ_9) \times \Gamma(\mZ_{m^2})$ admits DAML for all
prime $m > 3$.
\end{theorem}

\begin{proof}
It is observed that the graph $\Gamma(\mZ_9) \times \Gamma(\mZ_{m^2})
\cong P_2 \times K_{m-1}$ for all prime $m \geq 3$.

Already Sankar et al. \cite{ss} proved that, for all $n \geq 3$, $P_2 \times K_n$ admits DAML.

Therefore, the graph $\Gamma(\mZ_9) \times \Gamma(\mZ_{m^2})$ admits DAML
for all prime $m > 3$.
\end{proof}

\begin{theorem}
The graph $\Gamma(\mZ_6) \times \Gamma(\mZ_{m^2})$ admits DAML for all
prime $m \geq 3$.
\end{theorem}

\begin{proof}
Consider the zero-divisor graphs $\Gamma(\mZ_6)$ and
$\Gamma(\mZ_{m^2})$ for all prime \linebreak $m \geq 3$.  Let $G =
\Gamma(\mZ_6) \times \Gamma(\mZ_{m^2})$ with $|V(G)| = 3m-3$.

\noi{\bf Case (i):} Let $m = 3$.  The graph $\Gamma(\mZ_6) \times
\Gamma(\mZ_9)$ admits DAML with the labels as in Figure \ref{fig14}.  It is
clear that $\Gamma(\mZ_6) \times \Gamma(\mZ_9)$ admits DAML.

\begin{figure}[h!]
  \centerline{\includegraphics{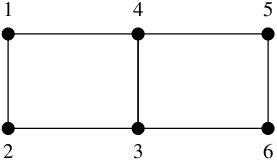}}
  \caption{The graph $\Gamma(\mZ_6) \times \Gamma(\mZ_9)$}
  \label{fig14}
\end{figure}

\noi{\bf Case (ii):} Let $m > 3$ and prime.  Consider the graph
$\Gamma(\mZ_6) \times \Gamma(\mZ_{m^2})$ with $|V(\Gamma(\mZ_6) \times
\Gamma(\mZ_{m^2}))| = 3m-3$.  We label the vertices of $\Gamma(\mZ_6)
\times \Gamma(\mZ_{m^2})$ as shown in Figure \ref{fig15}.

\begin{figure}[h!]
  \centerline{\includegraphics{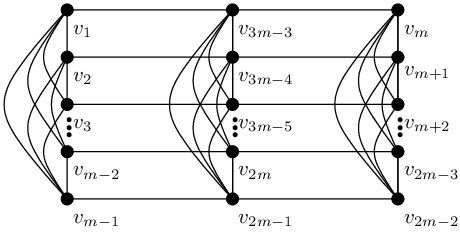}}
  \caption{The graph $\Gamma(\mZ_6) \times \Gamma(\mZ_{m^2})$}
  \label{fig15}
\end{figure}

Define a labeling function $f : V(G) \to \{1, 2, ..., 3m-3\}$ by\\
$f(v_i) = i$, if $1 \leq i \leq 3m-3$.

Then the weight of each vertex is given by:
\[w(v_i) = \begin{cases} \frac{1}{2} [m^2+5m-4]-2i, & \text{if } 1
  \leq i \leq m-1 \\ \frac{1}{2} [3m^2+3m-4]-2i, & \text{if } m \leq i
  \leq 2m-2 \\ \frac{1}{2} [5m^2-9m+4]-i+[3m-(2j-1)], & \text{if }
  2m-1 \leq i \leq 3m-3 \text{ and } j = 3+i-2m. \end{cases}\] 
It follows that $w(v_i) \neq w(v_j)$, for all $i \neq j$.\\Therefore,
the graph $\Gamma(\mZ_6) \times \Gamma(\mZ_{m^2})$ admits DAML for all $m >
3$.

Hence, the graph $\Gamma(\mZ_6) \times \Gamma(\mZ_{m^2})$ admits DAML for
all prime $m \geq 3$.
\end{proof}

\begin{theorem}
The graph $\Gamma(\mZ_9) \times \Gamma(\mZ_{3m})$ admits DAML for all odd
$m$ and $m \not\equiv 0 \ (\text{mod } 3)$.
\end{theorem}

\begin{proof}

Let $G = \Gamma(\mZ_9) \times \Gamma(\mZ_{3m})$ with $|V(G)| =
2m+2$ for all odd $m$ and $m \not\equiv 0 \ (\text{mod } 3)$.

We label the vertices of $G$ as given in Figure \ref{fig16}.

\begin{figure}[h!]
  \centerline{\includegraphics{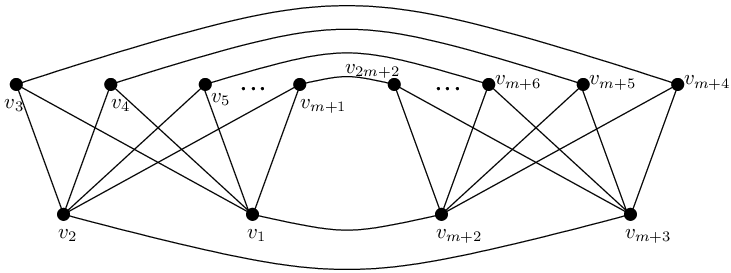}}
  \caption{The graph $\Gamma(\mZ_9) \times \Gamma(\mZ_{3m})$}
  \label{fig16}
\end{figure}

Let $f : V(G) \to \{1, 2, ..., 2m+2\}$ be a labeling function defined by\\
$f(v_i) = i$, for $1 \leq i \leq 2m+2$.

Then the weight of each vertex is given by:
\[w(v_i) = \begin{cases} \frac{1}{2} (m^2+5m), & \text{for }
  i=1 \\ \frac{1}{2} (m^2+5m+2), & \text{for } i=2 \\ m+i+4, & \text{for
  } 3 \leq i \leq m+1 \\ \frac{1}{2} (3m^2+3m-4), & \text{for } i=m+2
  \\ \frac{1}{2} (3m^2+3m-2), & \text{for } i=m+3 \\ m+i+4, & \text{for
  } m+4 \leq i \leq 2m+2 \end{cases}\]

It follows that $w(v_i) \neq w(v_j)$, for all $i \neq j$.\\  Hence, the
graph $\Gamma(\mZ_9) \times \Gamma(\mZ_{3m})$ admits DAML for all odd $m$
and\\ $m \not\equiv 0 \ (\text{mod } 3)$.
\end{proof}

\section{Conclusion}

In this paper, we found the distance antimagic labeling of $[m\Gamma(\mZ_9) + \Gamma(\mZ_4)] \times \Gamma(\mZ_9)$ for all $m \geq 1$, $\Gamma(\mZ_6) \times \Gamma(\mZ_{2m})$ for all primes $m \geq 3$ and  $\Gamma(\mZ_6) \times \Gamma(\mZ_{m^2})$ for all primes $m \geq 3$.  The distance antimagic labeling of  $\Gamma(\mZ_m) \times \Gamma(\mZ_{n})$ for all values of $m$ and $n$ and  $\Gamma(\mZ_p) \times [\Gamma(\mZ_m) \times \Gamma(\mZ_{n})]$ for all values of $p, m$ and $n$ are left for future research.

\end{document}